\documentclass[reqno]{amsart}
\usepackage{amssymb}

\newcommand{\qform}[1]{{\left\langle{#1}\right\rangle}}
\newcommand{\pform}[1]{{\langle\!\langle{#1}\rangle\!\rangle}}

\newcommand{\C}{{\mathcal C}}

\newcommand{\ot}{\otimes}
\newcommand{\io}{{}^\iota}
\newcommand{\la}{\lambda} 
\newcommand{\tra}{{\tr}_\star}

\DeclareMathOperator{\Ad}{Ad}
\DeclareMathOperator{\ad}{ad}

\DeclareMathOperator{\tr}{tr}

\DeclareMathOperator{\End}{End}

\DeclareMathOperator{\Int}{Int}

\newtheorem{prop}{Proposition}[section]
\newtheorem{lem}[prop]{Lemma}
\newtheorem{thm}[prop]{Theorem}

\theoremstyle{remark}
\newtheorem*{remark}{Remark}

\title[Quadratic forms of dimension 8]{Quadratic forms of dimension $8$ with trivial discriminant
  and Clifford algebra of index $4$.}

\author[A. Masquelein]{Alexandre Masquelein} 
\address{D\'epartement de Math\'ematiques\\ 
Universit\'e catholique de Louvain\\
Chemin du cyclotron, 2\\
B1348 Louvain-la-Neuve\\
Belgique}
\email{Alexandre.Masquelein@uclouvain.be}

\author[A. Qu\'eguiner-Mathieu]{Anne Qu\'eguiner-Mathieu} 
\address{Universit\'e Paris 13 (LAGA)\\
CNRS (UMR 7539)\\
Universit\'e Paris 12 (IUFM)\\
93430 Villetaneuse, France}
\email{queguin@math.univ-paris13.fr}
\urladdr{http://www-math.univ-paris13.fr/{\textasciitilde}queguin/}

\author[J.-P. Tignol]{Jean-Pierre Tignol}
\address{D\'epartement de Math\'ematiques\\ 
Universit\'e catholique de Louvain\\
Chemin du cyclotron, 2\\
B1348 Louvain-la-Neuve\\
Belgique}
\email{jean-pierre.tignol@uclouvain.be}
\urladdr{http://wwww.math.ucl.ac.be/membres/tignol}

\date{\today}
\thanks{The third author is supported in part by the F.R.S.--FNRS (Belgium)}

\begin{document}

\begin{abstract}
Izhboldin and Karpenko proved in~\cite[Thm 16.10]{IK00} that any
quadratic form of dimension $8$ with trivial discriminant and Clifford
algebra of index $4$ is isometric to the transfer, with respect to some quadratic
\'etale extension, of a quadratic form similar to a two-fold Pfister form. 
We give a new proof of this result, based on a theorem of
decomposability for degree $8$ and index $4$ algebras with orthogonal
involution.  
\end{abstract}

\maketitle

Let $WF$ denote the Witt ring of a field $F$ 
of characteristic different from $2$. 
As explained in~\cite[X.5 and XII.2]{Lam2}, one would like to describe 
those quadratic forms whose Witt class belongs to the 
$n$th power $I^nF$ of the fundamental ideal $IF$ of $WF$. 
By the Arason-Pfister Hauptsatz, such a form is hyperbolic if it has dimension $<2^n$
and similar to a Pfister form if it has dimension $2^n$. 
More generally, Vishik's Gap Theorem gives the possible dimensions of anisotropic
forms in $I^nF$. 

In addition, one may describe explicitly, for some small values of
$n$, low dimensional anisotropic quadratic forms in $I^nF$. 
This is the case, in particular, for $n=2$, that is for even-dimensional quadratic forms with trivial
discriminant. 
In dimension $6$, it is well known that such a form is similar to an
Albert form, and uniquely determined up to similarity by its Clifford
invariant. 
In dimension $8$, if the index of 
the Clifford algebra is $\leq 4$, Izhboldin and Karpenko proved in~\cite[Thm 16.10]{IK00} that 
it is isometric to the transfer, with respect to some quadratic
\'etale extension, of a quadratic form similar to a two-fold Pfister
form. 

The purpose of this paper is to give a new proof of Izhboldin and
Karpenko's result. 
Our proof is in the framework of
algebras with involution, and does not use Rost's description of
$14$-dimensional forms in $I^3F$ (see~\cite[Rmk
  16.11.2]{IK00}). 
More precisely, we use triality~\cite[(42.3)]{KMRT} to translate the
question into a question on algebras of degree $8$ and index $4$ with orthogonal
involution. Our main tool then is a decomposability theorem
(Thm.~\ref{dec.thm}), proven in \S~\ref{dec.sec}. 
We also use a refinement 
of a statement of Arason~\cite[4.18]{A} describing the even part of
the Clifford algebra of a
transfer (see Prop.~\ref{clif.prop} below). 

\section{Notations and statement of the theorem} 

Throughout the paper, we work over a base field $F$ of characteristic
different from $2$. We refer the reader to~\cite{KMRT} and~\cite{Lam2}
for background information on algebras with involution and on
quadratic forms. 
However, we depart from the notation in~\cite{Lam2} by
using $\pform{a_1,\dots,a_n}$ to denote the $n$-fold Pfister form
$\otimes_{i=1}^n\qform{1,-a_i}$. 
For any quadratic space $(V,\phi)$ over $F$, 
we let $\Ad_\phi$ be the algebra with
involution $(\End_F(V),\ad_\phi)$, where $\ad_\phi$ is the adjoint
involution with 
respect to $\phi$, denoted by $\sigma_\phi$ in~\cite{KMRT}. 

For any field extension $L/F$, we denote by $GP_n(L)$ the set of quadratic
forms that are similar to $n$-fold Pfister forms. This
notation extends to the quadratic \'etale extension $F\times F$ by
$GP_n(F\times F)=GP_n(F)\times GP_n(F)$. 
For any quadratic form $\psi$ over $L$, let $\C(\psi)$ be  
its full Clifford algebra, with even part $\C_0(\psi)$. 
Both $\C(\psi)$ and $\C_0(\psi)$ are endowed with a canonical
involution, which is the identity on the underlying vector space, denoted by $\gamma$
(see~\cite[p.89]{KMRT}). If $\psi$ has even dimension and
trivial discriminant, then its even Clifford algebra splits as a
direct product $\C_+(\psi)\times \C_-(\psi)$, for some isomorphic central simple algebras
$\C_+(\psi)$ and $\C_-(\psi)$ over $F$ (see~\cite[V, Thm 2.5]{Lam2}). 
Those algebras are Brauer-equivalent to the full Clifford algebra of
$\psi$ and their Brauer class is the Clifford invariant of $\psi$. 
Assume moreover that
$\dim(\psi)\equiv 0\bmod 4$. As explained in~\cite[(8.4)]{KMRT}, the involution $\gamma$ then induces an
involution on each factor of $\C_0(\psi)$, and one may easily check
that the isomorphism between the two factors described in the proof of~\cite[V, Thm
  2.5]{Lam2} preserves the involution, so that we
actually get a decomposition $(\C_0(\psi),\gamma)\simeq
(\C_+(\psi),\gamma_+)\times(\C_-(\psi),\gamma_-)$, with
$(\C_+(\psi),\gamma_+)\simeq(\C_-(\psi),\gamma_-)$. 

Let $L/F$ be a quadratic field extension. 
For any quadratic form $\psi$ over $L$, we let $\tra(\psi)$ be the transfer of
$\psi$, associated to the trace map $\tr:L\rightarrow F$, as defined in 
~\cite[VII.1.2]{Lam2}. 
This definition extends to the split \'etale case $L=F\times F$ 
and leads to $\tra(\psi,\psi')=\psi+\psi'$. 
On the other hand, for any algebra $A$ over
$L$, we let $N_{L/F}(A)$ be its norm, as defined in~\cite[\S
  3.B]{KMRT}. Recall that the Brauer class of $N_{L/F}(A)$ is the
corestriction of the Brauer class of $A$. 
Moreover, if $A$ is endowed with an involution of the first kind $\sigma$, then 
the tensor product $\sigma\ot\sigma$ restricts to an involution
$N_{L/F}(\sigma)$ on $N_{L/F}(A)$. We use the following notation: 
$N_{L/F}(A,\sigma)=(N_{L/F}(A),N_{L/F}(\sigma))$. 
In the split \'etale case, we get 
$N_{F\times
  F/F}((A,\sigma),(A',\sigma'))=(A,\sigma)\ot(A',\sigma')$
(see~\cite[\S 15.B]{KMRT}). 

\medskip
Let $(A,\sigma)$ be a degree $8$ algebra with orthogonal involution. 
We assume that $(A,\sigma)$ is {\em totally decomposable}, that is,
isomorphic to a tensor product of three
quaternion algebras with involution, 
\[(A,\sigma)=\ot_{i=1}^3(Q_i,\sigma_i).\] 
If $A$ is split (resp. has index $2$), then $(A,\sigma)$ admits a
decomposition as above in which
each quaternion algebra (resp. each but one) is split
(see~\cite{Becher}). Our main result is the following theorem: 
\begin{thm} 
\label{dec.thm}
Let $(A,\sigma)$ be a degree $8$ totally decomposable algebra with orthogonal involution. 
If the index of $A$ is $\leq 4$, then there exists $\lambda\in F^\times$
and a biquaternion
algebra with orthogonal involution $(D,\theta)$ such that 
\[(A,\sigma)\simeq (D,\theta)\ot\Ad_\pform{\lambda}.\]
\end{thm} 
The theorem readily follows from Becher's results mentioned above
if $A$ has index $1$ or $2$; it is proven in~\S~\ref{dec.sec}
for algebras of index $4$. 
For algebras of index $\leq 2$, we may even
assume that $(D,\theta)$ decomposes as a tensor product of two
quaternion algebras with involution; this is not the case anymore if
$A$ has index $4$, as was shown by Sivatski~\cite[Prop. 5]{Siv}. 

Using triality, we easily deduce the following from Theorem~\ref{dec.thm}:  
\begin{thm}[Izhboldin-Karpenko]
Let $\phi$ be an $8$-dimensional quadratic form over $F$. The
following are equivalent: 

(i) $\phi$ has trivial
discriminant and Clifford invariant of index $\leq 4$;

(ii) there exists a 
quadratic \'etale extension $L/F$ and a form $\psi\in GP_2(L)$ such that 
$\phi=\tra(\psi)$. 
\end{thm} 
If $\phi=\tra(\psi)$ for some $\psi\in GP_2(L)$, it follows from some
direct computation made in~\cite[\S 16]{IK00} that $\phi$ has trivial
discriminant and Clifford invariant of index $\leq 4$. 

Assume conversely that $\phi$ has trivial discriminant. 
By the Arason-Pfister Hauptsatz, $\phi$ is in $GP_3(F)$ if and only if
it has
trivial Clifford invariant. More generally, it is well-known that
$\phi$ decomposes as $\phi=\pform{a}q$ for some $a\in F^\times$ and some $4$-dimensional
quadratic form $q$ over $F$ if and only if its Clifford invariant has index $\leq
2$ (see for instance~\cite[Ex 9.12]{Kne2}).  
Hence, in both cases, $\phi$ decomposes as a sum $\phi=\pi_1+\pi_2$ of two forms
$\pi_1,\pi_2\in GP_2(F)$. 
This proves that condition (ii) holds with $L=F\times F$. 

In section~\ref{proof.section} below, we finish this proof by treating the index
$4$ case. This part of the proof differs from the argument given
in~\cite{IK00}. 
In particular, we do not use Rost's description of $14$-dimensional
forms in $I^3F$. 

\section{Clifford algebra of the transfer of a quadratic form} 

Let $L/F$ be a quadratic field extension.  
By Arason~\cite[4.18]{A}, for any quadratic form $\psi\in GP_2(L)$,
the Clifford invariant of the transfer $\tra(\psi)$ coincides with the
corestriction of the Clifford invariant of
$\psi$. 
In this section, we extend this result, taking into account the
algebras with involution rather than just the Brauer classes. 
More precisely, we prove: 
\begin{prop}
\label{clif.prop} 
Let $L=F[X]/(X^2-d)$ be a quadratic \'etale extension of $F$. 
Consider a quadratic
form $\psi$ over $L$ with $\dim(\psi)\equiv 0\bmod 4$ and $d_\pm(\psi)=1$, so
that its even Clifford algebra decomposes as 
\[(\C_0(\psi),\gamma)\simeq(\C_+(\psi),\gamma_+)\times(\C_-(\psi),\gamma_-),\mbox{
  with }(\C_+(\psi),\gamma_+)\simeq(\C_-(\psi),\gamma_-).\]   
For any $\la\in L^\times$ represented by $\psi$, the two components of
the even Clifford
algebra of the transfer of $\psi$ are both isomorphic to 
\[(\C_+(\tr_\star(\psi)),\gamma_+)\simeq
\Ad_\pform{-d N_{L/F}(\lambda)}\otimes N_{L/F}(\C_+(\psi),\gamma_+).\]
\end{prop} 

\begin{proof} 
In the split \'etale case $L=F\times
F$,  
the quadratic form $\psi$
is a
couple $(\phi,\phi')$ of two quadratic forms over $F$ with 
\[\dim(\phi)=\dim(\phi')\equiv 0\bmod4\quad\mbox{and}\quad
d_\pm(\phi)=d_\pm(\phi')=1\in F^\star/F^{\star 2}.\] 
Pick $\la$ and $\la'$ in $F$ respectively represented by $\phi$ and
$\phi'$; 
the norm $N_{F\times F/F}(\la,\la')$ is $\la\la'$. 
So the following lemma proves the proposition in that case: 
\begin{lem} 
\label{etale.lem} 
Let $\phi$ and $\phi'$ be two quadratic forms over 
$F$ of the same dimension $n\equiv 0\bmod4$ and trivial discriminant. 
For any $\la$ and $\la'\in F^\times$, respectively represented by
$\phi$ and $\phi'$, 
the components of the even Clifford algebra of
the orthogonal sum $\phi+\phi'$ are isomorphic to 
\[(\C_+(\phi+\phi'),\gamma_+)\simeq\Ad_\pform{-\la\la'}\otimes(\C_+(\phi),\gamma_+)\otimes(\C_+(\phi'),\gamma_+).\]
\end{lem} 
\begin{proof}[Proof of Lemma~\ref{etale.lem}]

Denote by $V$ and $V'$ the underlying quadratic spaces. 
The natural embeddings $V\hookrightarrow V\oplus V'$ and 
$V'\hookrightarrow V\oplus V'$ induce $F$-algebra homomorphisms 
\[\C(\phi)\rightarrow\C(\phi+\phi')\mbox{ and }\C(\phi')\rightarrow \C(\phi+\phi').\] 
One may easily check that the images of the even parts
centralize each other, so that we get an $F$-algebra homomorphism 
\[(\C_0(\phi),\gamma)\otimes(\C_0(\phi'),\gamma)\rightarrow
(\C_0(\phi+\phi'),\gamma).\] 
Pick orthogonal bases $(e_1,\dots,e_n)$ of $(V,\phi)$ and
$(e'_1,\dots,e'_n)$ of $(V',\phi')$. 
The basis of $\C_0(\phi+\phi')$ consisting of products of an even
number of vectors of the set $\{e_1,\dots,e_n,e'_1,\dots,e'_n\}$ as described
in~\cite[V, cor 1.9]{Lam2}   
clearly contains the image of a basis of $\C_0(\phi)\otimes\C_0(\phi')$, so that
the homomorphism above is injective. 
In the sequel, we will identify $\C_0(\phi)$ and $\C_0(\phi')$ with
their images in $\C_0(\phi+\phi')$. 

Consider the element $z=e_1\dots e_n\in \C_0(\phi)$. 
As explained in~\cite[V, Thm2.2]{Lam2}, 
for any $v\in V$, one has $vz=-zv\in\C(\phi)$ and $z$ generates the
center of $\C_0(\phi)$. 
Since $\phi$ has dimension $0\bmod4$ and trivial discriminant, this
element $z$ is $\gamma$-symmetric, and multiplying $e_1$ by a scalar
if necessary, we may assume $z^2=1$. 
The two components of $\C_0(\phi)$ are 
$\C_+(\phi)=\C_0(\phi)\frac{1+ z}{2}$ and $\C_-(\phi)=\C_0(\phi)\frac{1-z}{2}$. 
Consider similarly $z'=e'_1\dots e'_n$, with $\gamma(z')=z'$ and
assume $z'^2=1$. 
The product $zz'$ also has square $1$ and generates the center of $\C_0(\phi+\phi')$. 
We denote by $\varepsilon$ the idempotent
$\varepsilon=\frac{1+zz'}{2}$, so that  
$\C_+(\phi+\phi')=\C_0(\phi+\phi')\varepsilon$ and
$\C_-(\phi+\phi')=\C_0(\phi+\phi')(1-\varepsilon)$. 

Let us now fix two vectors $v\in V$ and $v'\in V'$ such that 
$\phi(v)=\la$ and $\phi'(v')=\la'$. 
Since $\frac{1+z}{2}v^{-1}=v^{-1}\frac{1-z}{2}$, we have $vxv^{-1}\in
\C_-(\phi)$ for any $x\in\C_+(\phi)$. 
Using this identification between the two components, we may
diagonally embed $\C_+(\phi)$ in $\C_0(\phi)$ by considering  
$x\in\C_+(\phi)\mapsto x+vxv^{-1}\in\C_0(\phi).$
Similarly, we may embed $\C_+(\phi')$ in $\C_0(\phi')$ by 
$x'\in\C_+(\phi')\mapsto x'+v'x'v'^{-1}\in\C_0(\phi').$ 
Combining those two maps with the morphism
\[\C_0(\phi)\otimes\C_0(\phi')\rightarrow \C_0(\phi+\phi'),\] and the
projection \[y\in \C_0(\phi+\phi')\mapsto
y\varepsilon\in\C_+(\phi+\phi'),\] we get an algebra
homomorphism 
\[\begin{array}{ccc} 
\C_+(\phi)\otimes\C_+(\phi')&\rightarrow&\C_+(\phi+\phi'),\\
x\otimes x'&\mapsto&(x+vxv^{-1})(x'+v'x'v'^{-1})\varepsilon.\\
\end{array}\]

One may easily check on generators that this map is not trivial; hence
it is injective. To conclude the proof, it only remains to identify the
centralizer of the image, which by dimension count has degree $2$. 
It clearly contains 
$\frac{z+z'}{2}\varepsilon$ and $vv'\varepsilon$. 
Moreover, these two elements anticommute, have square $\varepsilon$
and $-\la\la'\varepsilon$, and are respectively symmetric and skew-symmetric under $\gamma$. Hence they generate a split quaternion
algebra, with orthogonal involution of discriminant $-\lambda\la'$,
which is isomorphic to $\Ad_\pform{-\la\la'}$. 
\end{proof} 
\medskip

This concludes the proof in the split \'etale case. 
Until the end of this section, we assume $L$ is a quadratic field
extension of $F$, with non-trivial $F$-automorphism denoted by $\iota$. 
To prove the proposition in this case, we will use the following
description of the transfer of a quadratic form and its Clifford
algebra. 

Let $\psi$ be any quadratic form
over $L$, defined on the vector space $V$. We
consider its conjugate ${}^\iota V=\{{}^\iota v,\ v\in V\}$ with the
following operations ${}^\iota v_1+{}^\iota v_2={}^\iota (v_1+v_2)$
and $\lambda\cdot{}^\iota v={}^\iota (\iota(\lambda)\cdot v)$,
for any $v_1,v_2$ and $v$ in $V$ and $\lambda\in L$. Clearly,
${}^\iota\psi({}^\iota v)=\iota(\psi(v))$ is a quadratic form on
${}^\iota V$. 
One may easily check from the definition given in~\cite[VII \S1]{Lam2} that
the quadratic form $\tra(\psi)$ is nothing but the restriction 
of $\psi+{}^\iota\psi$ to the $F$-vector space of fixed points
$(V\oplus{}^\iota V)^{s}$, where $s$ is the switch semi-linear automorphism defined on the direct sum
$V\oplus{}^\iota V$ by $s(v_1+{}^\iota v_2)=v_2+{}^\iota v_1$. 

Moreover, $s$ 
induces a semi-linear automorphism of order $2$ of the tensor algebra 
$T(V\oplus\io V)$ which preserves the ideal generated by 
the elements \[(v_1+\io v_2)\ot(v_1+\io v_2)-\big(\psi(v_1)+\io\psi(\io
v_2)\big).\] Hence, we get a semi-linear automorphism $s$ of order $2$
on the Clifford algebra $\C(\psi+\io\psi)$, which commutes with the
canonical involution. The set of fixed points
$\big(\C(\psi+\io\psi)\big)^s$ is an $F$-algebra; the involution
$\gamma$ restricts to an $F$-linear involution which we denote by $\gamma_s$. 
We then have: 
\begin{lem}
\label{clif.lem}
The natural embedding $(V\oplus\io V)\hookrightarrow
\C(\psi+\io\psi)$, restricted to $(V+\io V)^s$, 
induces an isomorphism of graded algebras 
\[\big(\C(\tra(\psi)),\gamma\big)\tilde\rightarrow\big( (\C(\psi+\io\psi))^s,\gamma_s).\]  
\end{lem} 
\begin{proof}[Proof of Lemma~\ref{clif.lem}]
The natural embedding $(V\oplus\io V)\hookrightarrow
\C(\psi+\io\psi)$ restricts to an injective map
$i:\,(V+\io V)^s\hookrightarrow \C(\psi+\io\psi)^s$, which clearly
satisfies 
\[i(w)^2=(\psi+\io\psi)(w)\quad\mbox{for any }w\in(V\oplus\io V)^s.\] 
By the universal property of Clifford algebras, it extends to a non-trivial algebra homomorphism $\C(\tra(\psi))\mapsto \C(\psi+\io\psi)^s$, which
clearly preserves the grading.  
Since $\C(\tra(\psi))$ is simple, and both algebras have the same
dimension, it is an isomorphism. 
Clearly, $\gamma$ coincides with $\gamma_s$ under this isomorphism. 
\end{proof} 

Hence, we want to describe one component of
$\C_0(\tr_\star(\psi))\simeq (\C_0(\psi+\io\psi))^s$. 
We proceed as in the split \'etale case. 
Fix an orthogonal basis $e_1,\dots e_n$ of $V$ over
$L$ such that $\psi(e_n)=\lambda$. The elements $\io e_1,\dots,\io
e_n$ are an orthogonal basis of $\io V$ and $\io\psi(\io
e_n)=\iota(\lambda)$. 
We may moreover assume that $z=e_1\dots e_n$ and $\io z=\io
e_1\dots\io e_n$ have square $1$. 
Since the idempotent $\varepsilon=\frac{1+z\io z}{2}\in
C_0(\psi+\io\psi)$  
satisfies $s(\varepsilon)=\varepsilon$, the semilinear automorphism
$s$ preserves each factor $\C_+(\psi+\io\psi)$ and
  $\C_-(\psi+\io\psi)$. Hence, the components of
$\C_0(\tr_\star(\psi))$ are 
\[\C_0(\tr_\star(\psi))=(\C_+(\psi+\io\psi))^s\times (\C_-(\psi+\io\psi))^s.\]
Moreover, by Lemma~\ref{etale.lem}, we have 
\[\C_+(\psi+\io\psi)\simeq\Ad_\pform{-\la\iota(\la)}\otimes(\C_+(\psi),\gamma)\otimes
(\C_+(\io\psi),\gamma),\]
and it remains to understand the action of the switch automorphism on
this tensor product. 
First, one may identify $\C_+(\io\psi)$ with the algebra
$\io\C_+(\psi)$ defined by 
\[\io\C_+(\psi)=\{\io x,\ x\in\C_+(\psi)\},\] with the
operations 
\[\io x+\io y=\io(x+y),\ \ \io x\io y=\io(xy)\mbox{ and }\io(\la
x)=\iota(\la)\io x,\] for all $x,y\in\C_+(\psi)$ and $\la\in L$.  
Clearly, the switch automorphism acts on the tensor product
\[\C_+(\psi)\otimes
\C_+(\io\psi)\simeq \C_+(\psi)\otimes
\io\C_+(\psi),\]
by \[s(x\otimes\io y)=y\otimes \io x,\] 
and by definition of the corestriction (see \cite[3.B]{KMRT}), the
$F$-subalgebra of fixed points is 
\[\big((\C_+(\psi),\gamma)\otimes
(\io\C_+(\psi),\gamma)\big)^s=N_{L/F}(\C_+(\psi),\gamma).\]

It remains to understand the action of the switch on the centralizer, 
which is the split quaternion algebra over $L$ generated
by $x=\frac{z+\io z}{2}\varepsilon$ and $y=e_n\io e_n\varepsilon$. 
The element $x$ clearly is $s$-symmetric, while $y$ satisfies
$s(y)=-y$. 
Let $\delta$ be a generator of the quadratic extension $L/F$, so that
$\iota(\delta)=-\delta$ and $\delta^2=d$. 
Since the switch map $s$ is $L/F$ semi-linear,
we may replace $y$ by $\delta y$ which now satisfies $s(\delta
y)=\delta y$. 
Hence, the set of fixed
points under $s$ is the split quaternion
algebra over $F$ generated by $x$ and $\delta y$. 
Since $(\delta y)^2=-dN_{L/F}(\la)$, it is 
isomorphic to $\Ad_\pform{-dN_{L/F}(\la)}$. 
\end{proof}

\section{Proof of the decomposability theorem} 
\label{dec.sec} 

In this section, we finish the proof of Theorem~\ref{dec.thm}. 
Let $(A,\sigma)=\otimes_{i=1}^3(Q_i,\sigma_i)$ be a product of three
quaternion algebras with orthogonal involution. 
We assume that $A$ has index $4$, so that it is Brauer-equivalent to 
a biquaternion division
algebra $D$. 
We have to prove that $(A,\sigma)$ is isomorphic to 
$(D,\theta)\otimes\Ad_\pform{\la}$ for a well chosen involution
$\theta$ on $D$ and some $\la\in F^\times$. 

The algebra $D$ is endowed with an orthogonal involution
$\tau$, and we may represent 
\[(A,\sigma)=(\End_D(M),\ad_h),\] for some $2$-dimensional hermitian
module $(M,h)$ over $(D,\tau)$. 
Let us consider a diagonalisation $\qform{a_1,a_2}$ of $h$, and
define \[\theta=\Int(a_1^{-1})\circ\tau.\] 
With respect to this new involution, we get another representation 
\[(A,\sigma)=(\End_D(M),\ad_{h'}),\] where $h'$ is a hermitian form 
over $(D,\theta)$ which diagonalises as $h'=\qform{1,-a}$ for some
$\theta$-symmetric element $a\in D^\times$. The theorem now follows
from the following lemma: 
\begin{lem} \label{inv.lem} 
The involutions $\theta$ and $\theta'=\Int(a^{-1})\circ\theta$ of the
biquaternion algebra $D$ are
conjugate.  
\end{lem} 
Indeed, assume there exists $u\in A^\times$ such that 
$\theta=\Int(u)\circ\theta'\circ\Int(u^{-1})$. 
We then have $\theta=\Int(ua^{-1})\circ\theta\circ\Int(u^{-1})
=\theta\circ \Int(\theta(u)^{-1}au^{-1})$. 
Hence, there exists $\la\in F^\times$ such that 
$\theta(u)^{-1}au^{-1}=\la$, that is $a=\lambda\theta(u)u$. 
This implies that the hermitian form $h'=\qform{1,-a}$ over
$(D,\theta)$ is isometric to $\qform{1,-\la}$. 
Since $\la\in F^\times$, we get 
$(A,\sigma)=(\End_D(M),\ad_{\qform{1,-\la}})=(D,\theta)\otimes\Ad_\pform{\la}$,
and it only remains to prove the lemma.    

\begin{proof}[Proof of Lemma~\ref{inv.lem}] 

We want to compare the orthogonal involutions $\theta$ and $\theta'$
of the biquaternion algebra $D$. 
By~\cite[Prop. 2]{LT}, they are conjugate if and only if their
Clifford algebras $\C$ and $\C'$ are isomorphic as $F$-algebras. 
This can be proven as follows. 

Since $(A,\sigma)$ is a product of three quaternion algebras with
involution, we know 
from~\cite[(42.11)]{KMRT} that the discriminant of $\sigma$ is $1$ and
its Clifford algebra has one split component. 

On the other hand, the representation $(A,\sigma)=(\End_D(M),\ad_{\qform{1,-a}})$
tells us that $(A,\sigma)$ is an orthogonal sum, as in~\cite{D}, of $(D,\theta)$ and
$(D,\theta'$). 
Hence its invariants can be computed in terms of those of $(D,\theta)$
 and $(D,\theta')$. 
By \cite[Prop. 2.3.3]{D}, the discriminant of $\sigma$ is the
product of the discriminants of $\theta$ and $\theta'$. 
So $\theta$ and $\theta'$ have the same discriminant, and we may
identify the centers $Z$ and $Z'$ of their Clifford algebras in two
different ways. 
We are in
the situation described in~\cite[p. 265]{LT}, where the Clifford
algebra of such an orthogonal sum is computed. 
In particular, since one component of the Clifford algebra of
$(A,\sigma)$ is split, it follows from~\cite[Lem 1]{LT} that 
\[\C\simeq\C'\qquad\mbox{ or }\qquad\C\simeq\io\C',\]
depending on the chosen identification between $Z$ and $Z'$. 
In both cases, $\C$ and $\C'$ are isomorphic as $F$-algebras, and this concludes
the proof.  
\end{proof}

\section{A new proof of Izhboldin and Karpenko's theorem} 
\label{proof.section} 

Let $\phi$ be an $8$-dimensional quadratic form over $F$ with trivial
discriminant and Clifford invariant of index $4$. 
We denote by $(A,\sigma)$ one component of its even Clifford algebra, 
so that 
\[(\C_0(\phi),\gamma)\simeq (A,\sigma)\times (A,\sigma),\]
where $A$ is an index $4$ central simple algebra over $F$, with
orthogonal involution $\sigma$.  

By triality~\cite[(42.3)]{KMRT}, the involution $\sigma$ has trivial discriminant and its Clifford
algebra is
\[\C(A,\sigma)=\Ad_\phi\times (A,\sigma).\]
In particular, it has a split component, so that the algebra with involution $(A,\sigma)$ is
isomorphic to a tensor product of three quaternion algebras with
involution~(see \cite[(42.11)]{KMRT}). 
Hence we can apply our decomposability theorem~\ref{dec.thm}, 
and write $(A,\sigma)=(D,\theta)\otimes \Ad_\pform{\la}$ for some
biquaternion division algebra with orthogonal involution $(D,\theta)$
and some $\la\in F^\times$. 

Let us denote by $d$ the discriminant of $\theta$, and 
let $L=F[X]/(X^2-d)$ be the corresponding quadratic \'etale
extension. Consider the image $\delta$ of $X$ in $L$. 
By Tao's computation of the Clifford
algebra of a tensor product~\cite[Thm. 4.12]{Tao:Clifford}, the components of 
$\C(A,\sigma)$ are Brauer-equivalent to the quaternion algebra
$(d,\la)$ over $F$ and the tensor product
$(d,\la)\otimes A$. 
Since $A$ has index $4$, the split component has to be
$(d,\la)$, so that $\la$ is a norm of $L/F$, say $\la=N_{L/F}(\mu)$. 

Consider now the Clifford algebra of $(D,\theta)$. It is a quaternion
algebra $Q$ over $L$, endowed with its canonical (symplectic)
involution $\gamma$. Denote by $n_Q$ the norm form of $Q$, that is
$n_Q=\pform{\alpha,\beta}$ if $Q=(\alpha,\beta)_L$. 
It is a $2$-fold Pfister form and for any $\ell\in L^\star$, 
$(\C_+(\qform{\ell}n_Q),\gamma_+)\simeq (Q,\gamma)$. 
Moreover, by the equivalence of categories $A_1^2\equiv D_2$ described
in~\cite[(15.7)]{KMRT}, the algebra with involution  
$(D,\theta)$ is canonically isomorphic to $N_{L/F}(Q,\gamma)$. 

Hence we get that
$(A,\sigma)=N_{L/F}(Q,\gamma)\otimes\Ad_\pform{-dN_{L/F}(\delta\mu)}$. 
By Proposition~\ref{clif.prop}, this implies that 
\[(A,\sigma)\times(A,\sigma)\simeq(\C_0(\tr_\star(\psi)),\gamma),\]
where $\psi=\qform{\delta\mu}n_Q$. 
Applying again triality~\cite[(42.3)]{KMRT}, we get that the split
component $\Ad_\phi$ of the Clifford algebra of $(A,\sigma)$ also is
isomorphic to $\Ad_{\tr_\star(\psi)}$, so that the quadratic
forms $\phi$ and $\tr_\star(\psi)$ are similar. 
This concludes the proof since $\psi$ belongs to $GP_2(L)$. 

\begin{remark}
Let $\phi$ and $(A,\sigma)$ be as above, and 
let $L=F[X]/(X^2-d)$ be a fixed quadratic \'etale extension of $F$. 
It follows from the proof that the quadratic form $\phi$ is isometric to
the transfer of a form $\psi\in GP_2(L)$ if and only if $(A,\sigma)$ admits a decomposition
$(A,\sigma)=\Ad_\pform{\la}\otimes (D,\theta)$, with
$d_\pm(\theta)=d$. 
In particular, the quadratic form $\phi$ is a sum of two
forms similar to $2$-fold Pfister forms exactly when the algebra with
involution $(A,\sigma)$ admits a decomposition as
$(D,\theta)\otimes\Ad_\pform{\la}$ with $\theta$ of discriminant $1$,
that is when it decomposes as a tensor product of three quaternion
algebras with involution, with one split factor. 

Such a decomposition does not always exist, as was shown by
Sivatski~\cite[Prop 5]{Siv}. 
This reflects the fact that $8$-dimensional quadratic forms $\phi$ 
with trivial discriminant and Clifford algebra of index $\leq 4$ 
do not always decompose as a sum of two forms similar to two-fold
Pfister forms (see~\cite[\S 16]{IK00} and~\cite{HT} for explicit examples). 
\end{remark}

\providecommand{\bysame}{\leavevmode ---\ }
\providecommand{\og}{``}
\providecommand{\fg}{''}
\providecommand{\smfandname}{et}
\providecommand{\smfedsname}{\'eds.}
\providecommand{\smfedname}{\'ed.}
\providecommand{\smfmastersthesisname}{M\'emoire}
\providecommand{\smfphdthesisname}{Th\`ese}

\end{document}